\documentclass[10pt]{amsart}
\usepackage{verbatim}
\usepackage{eucal,url,amssymb,stmaryrd,enumerate,amscd}
\usepackage{amsmath}
\usepackage[pagebackref,colorlinks=true]{hyperref}
\usepackage[pagebackref]{hyperref}
\usepackage{amsfonts}
\usepackage{amsthm}
\usepackage[margin=.9in]{geometry}

\numberwithin{equation}{section}
\newtheorem{thrm}{Theorem}[section]
\newtheorem{lemma}[thrm]{Lemma}
\newtheorem{prop}[thrm]{Proposition}
\newtheorem{cor}[thrm]{Corollary}

\newtheorem{rmrk}[thrm]{Remark}

\newtheorem{conv}[thrm]{Convention}
\newtheorem{conj}[thrm]{Conjecture}

\newcommand{\vol}{\, Vol_{\theta}}

\def\gr{\nabla f}
\def\bi{\nabla}

\begin{document}

\begin{abstract}
We prove a CR version of the Obata's  result  for the first
eigenvalue of the sub-Laplacian in the setting of a compact
strictly pseudoconvex pseudohermitian three dimensional manifold
with non-negative CR-Panietz operator which satisfies a
Lichnerowicz type condition. We show that if the first positive
eigenvalue of the sub-Laplacian takes the smallest possible value
then, up to a homothety of the pseudohermitian structure, the
manifold is the standard Sasakian three dimensional unit sphere.
\end{abstract}

\title[An Obata-type theorem  on
a three-dimensional CR manifold]{An Obata-type theorem on a
three-dimensional CR manifold}
\date{\today }
\author{S. Ivanov}
\address[Stefan Ivanov]{University of Sofia, Faculty of Mathematics and
Informatics, blvd. James Bourchier 5, 1164, Sofia, Bulgaria}
\email{ivanovsp@fmi.uni-sofia.bg}
\author{D. Vassilev}
\address[Dimiter Vassilev]{ Department of Mathematics and Statistics\\
University of New Mexico\\
Albuquerque, New Mexico, 87131-0001}
\email{vassilev@math.unm.edu}
\maketitle
\tableofcontents


\setcounter{tocdepth}{2}

\section{Introduction}

The classical theorems of Lichnerowicz \cite{Li} and Obata \cite{O3} give
correspondingly a lower bound for the first eigenvalue of the Laplacian on a
compact manifold with a lower Ricci bound and characterize the case of
equality.

A CR analogue of the Lichnerowicz theorem was found by Greenleaf
\cite{Gr} for dimensions $2n+1>5$, while the corresponding results
for  $n=2$ and  $n=1$ were achieved later in \cite{LL}  and
\cite{Chi06}, respectively. As a continuation of this line of
results in the setting of geometries modeled on the rank  one
symmetric spaces in \cite{IPV1} it was proven a  quaternionic
contact version of the Lichnerowicz result.

{ The CR Lichnerowicz type result states that on a compact
$2n+1$-dimensional strictly pseudoconvex pseudohermitian manifold
satisfying a certain positivity condition the first eigenvalue of
the sub-Laplacian is grater than or equal to the first eigenvalue
of the standard Sasakian sphere}.   Greenleaf \cite{Gr} showed the
result for $n\geq 3$, while S.-Y. Li and H.-S. Luk adapted in
\cite{LL} Greenleaf's proof to cover the case $n=2$. They also
gave a version of the case $n=1$ assuming further a condition on
the covariant derivative of the pseudohermitian torsion. H.-L.
Chiu in \cite{Chi06} gave a  version in dimension three assuming
in addition that the CR-Paneitz operator is non-negative. We
remark that if $n>1$ the CR-Paneitz operator is always
non-negative,  while in the case $n=1$ the vanishing of the
pseudohermitian torsion implies that the CR-Paneitz operator is
non-negative,  see \cite{Chi06} and\cite{CCC07}.  Further results
in the CR case  appeared in  \cite{CC07,CC09b,CC09a}, \cite{Bar}
and \cite{ChW} adding a corresponding inequality for $n=1$, or
characterizing the equality case in the vanishing pseudohermitian
torsion case (the Sasakian case).

The problem of the existence of an Obata-type theorem in pseudohermitian
manifold was considered in \cite{CC09a} where the following CR analogue of
Obata's theorem was conjectured.

\begin{conj}[\protect\cite{CC09a}]
\label{conj1} Let $(M,\theta)$ be a closed pseudohermitian (2n+1)-manifold
with $n\ge 2$. In addition we assume the CR-Paneitz operator is nonnegative if
n = 1. Suppose there is a positive constant $k_0$ such that the
pseudohermitian Ricci curvature $Ric$ and the pseudohermitian torsion $A$
satisfy the inequality \eqref{condm}. If $\frac{n}{n+1}k_0$ is an eigenvalue
of the sub-Laplacian then $(M,\theta)$ is the standard (Sasakian) CR
structure on the unit sphere in $\mathbb{C}^{n+1}$. %
\end{conj}
This conjecture was proved in the case of a vanishing
pseudohermitian torsion (Sasakian case) in \cite{CC09a} for $n\ge
2$ and in \cite{CC09b} for $n=1$.  The non-Sasakian case was
considered in \cite{ChW12} where Conjecture~\ref{conj1} was
established under some assumptions on the divergence and
the second covariant derivative of the pseudohermitian torsion.

A dimension independent  proof of the results due to Greenleaf,
S.-Y. Li \& H.-S. Luk, and H.-L. Chiu   based on the
non-negativity of the CR-Paneitz operator can be found in the
Appendix of \cite{IVO}. The key to this direct exposition of the
known results is the last inequality in the proof of
\cite[Theorem~8.8]{IVO} which states that for any smooth function
$f$ on {a compact pseudohermitian manifold $(M,\theta)$
satisfying \eqref{condm}} we have
\begin{equation} \label{e:obata ineq}
0\geq\int_M\Big[\left (-\frac {n+1}{n}\lambda+k_0\right )|\gr|^2
+ \left
|(\nabla^2f)_{[-1]}\right |^2 -\frac {3}{2n}P_f(\gr)\Big]\vol
\end{equation}
 Here, following \cite{L1,GL88} for a given function $f$ we define the one form,
\begin{equation}  \label{e:Pdef}
P(X)\equiv P_{f}(X)=\nabla ^{3}f(X,e_{b},e_{b})+\nabla
^{3}f(JX,e_{b},Je_{b})+4nA(X,J\nabla f)
\end{equation}%
and also  a fourth order differential operator  (the so called CR-Paneitz operator in \cite{Chi06}),
\begin{equation}  \label{e:Cdef}
Cf=-\nabla ^{\ast }P=(\nabla_{e_a} P)({e_a})=\nabla
^{4}f(e_a,e_a,e_{b},e_{b})+\nabla
^{4}f(e_a,Je_a,e_{b},Je_{b})-4n\nabla^* A(J\nabla
f)-4n\,g(\nabla^2 f,JA),
\end{equation}
{where $\{e_1,\dots,e_{2n}\}$ is an orthonormal basis and a
summation over repeated indices is understood.}

Taking into account the divergence formula on a compact
pseudohermitian manifold,  the non-negativity condition of the
Paneitz operator means that we have
\begin{equation*}
\int_M f\cdot Cf \vol=-\int_MP_f(\gr) \vol \geq 0
\end{equation*}
for any smooth function $f$. In the three dimensional case this
condition is a CR invariant since it is independent of the choice
of the contact form. This follows from the conformal invariance of
$C$ proven in \cite{Hi93}.

A new method to attack the problem was developed by the authors in
\cite{IVO} where Conjecture~\ref{conj1} was proved under the
additional assumption of a divergence-free torsion. The new
approach of \cite{IVO} is based on the explicit form of the
Hessian with respect to the Tanaka-Webster connection of an
extremal eigenfunction $f$, i.e., an eigenfunction with eigenvalue
$n/(n+1)k_0$, and the formula for the pseudohermitian curvature.
Specifically, in the extremal case inequality \eqref{e:obata ineq}
is used in \cite{IVO} to determine, among other things, the
horizontal Hessian of an extremal eigenfunction $f, \triangle
f=\frac{n}{n+1}k_0f$, which after a rescaling can be put in the
form \cite{IVO}
\begin{equation}\label{e:hessian}
 \nabla ^{2}f(X,Y)=-fg(X,Y)-df(\xi )\omega (X,Y), \qquad X, Y \in H=Ker\, \theta.
\end{equation}
{The new idea in \cite{IVO} is  to determine the full
Hessian with respect to the Tanaka-Webster connection based on
\eqref{e:hessian}.}  One of the notable consequences of this is
the elliptic equation satisfied by the extremal first
eigenfunction, which allows the use of Riemannian unique
continuation results. This fact was later used in \cite{LW} where
the divergence-free condition of \cite{IVO} was shown to be
superfluous for the results of \cite{IVO} to hold true. In fact,
in this very recent paper \cite{LW} S.-Y. Li and X. Wang established
Conjecture~\ref{conj1} for $n>1$ completing our approach
\cite{IVO} with the introduction of a new integration by parts
idea, cf. \cite[Lemma 4]{LW}, involving suitable powers of the
extremal function.
\begin{rmrk}
The authors of \cite{LW} did not credit the paper \cite{IVO} for
the similar method and a number of  results in \cite{IVO} which
were used or reproved in \cite{LW} using  complex notation. For
example, the formulas in the crucial \cite[Proposition 4]{LW} are
stated in \cite[Lemma~3.1]{IVO}, \cite[Lemma~4.3]{IVO}, equation
(4.4) in \cite{IVO}, and the last equation in the proof of
 \cite[Lemma~4.1]{IVO}.  In addition, \cite[Lemma 2 ]{LW} follows directly
 (is a special case of) of equations (4.2) and (4.3) in \cite{IVO}. Nevertheless,
\cite{LW} contains a new and a very important step, namely
\cite[Lemma 4]{LW}, which makes possible the above explained
improvement of the result of \cite{IVO} in the case $n>1$.
\end{rmrk}

In the three dimensional case, \cite[Proposition 5]{LW} is not
correctly proved and a correct proof is contained in Section~6.2
of \cite{IVO}.  Furthermore, the proof of Conjecture~\ref{conj1}
presented in \cite{LW} for dimension three has a gap since formula
(4.8) in \cite{LW} does not hold in dimension three which
{prevents the use }of Lemma~3 and equality (4.3). Therefore
\cite[Corollary~1]{LW} can not be applied in the three dimensional
case.  The purpose of this paper is to settle Conjecture~\ref{conj1}
using dimension three where we  prove the following Theorem.

\begin{thrm}\label{main2} Let $(M, \theta)$ be a compact strictly pseudoconvex
pseudohermitian CR manifold of dimension three with a non-negative
CR-Paneitz operator. Suppose there is a positive constant $k_0$
such that the pseudohermitian Ricci curvature $Ric$ and the
pseudohermitian torsion $A$ satisfy the inequality
\begin{equation}  \label{condm}
Ric(X,X)+ 4A(X,JX)\geq k_0 g(X,X), \qquad X \in H=Ker\, \theta.
\end{equation}
If  $\lambda=\frac{1}{2}k_0$ is an eigenvalue of the
sub-Laplacian,  then up-to a scaling of $\theta$ by a positive
constant $(M,\theta)$ is the standard (Sasakian) CR structure on
the unit three-dimensional sphere in $\mathbb{C}^{2}$.
\end{thrm}
The value of the scaling is determined, for example, by the fact
that the standard pseudohermitian structure on the unit sphere has
first eigenvalue equal to $2$. The corresponding eigenspace is
spanned by the restrictions of all linear functions to the sphere.

The proof or Theorem \ref{main2} is based on the explicit form of
the Hessian \cite{IVO} with respect to the Tanaka-Webster
connection of an extremal eigenfunction $f$   and the integration
by parts involving powers of the extremal eigenfunction introduced
in \cite{LW}. {After these initial steps we prove
Theorem~\ref{main2} as a consequence of Theorem~\ref{t:A vansihes
if div 3D} taking into account the already established CR Obata
theorem  for pseudohermitian manifold with a vanishing
pseudohermitian torsion. Thus, the new result here is Theorem
\ref{t:A vansihes if div 3D} which shows that if on a three
dimensional compact  pseudohermitian manifold satisfying
\eqref{condm} and having, further, non-negative Paneitz operator
we have an eigenfunction $f$ with a horizontal Hessian given by
the above formula \eqref{e:hessian},  then the pseudohermitian
torsion vanishes, i.e., we have a Sasakian structure.} The new
idea in dimension three is to compare the calculated in \cite{IVO}
Ricci tensor with the Lichnerowicz type assumption \eqref{condm}
which results in the formula  for the full Hessian with respect to
the Tanaka-Webster connection of an extremal eigenfunction
expressed in Lemma~\ref{neww1}.

\begin{rmrk}
Following \cite{IV3}, a correction of
 the results of \cite{LW} in the three dimensional
case appeared in  \cite{LW1}. The correct argument uses the above
mentioned idea of \cite{IVO} which allows the "recovery" of
formula (5.2) from \cite{IV3}  for the full Hessian of an extremal
eigenfunction in dimension three.  This is a crucial fact in the
proofs of \cite[Theorem~9 \& Theorem~10]{LW1}. In the higher
dimensional case, \cite[Theorem~5 \& Theorem~8]{LW1} appeared
earlier in \cite{IVO} under the additional assumption that the
torsion is divergence free, see \cite[Theorem~1.3]{IVO}. In
\cite{LW1} the  "novelty" in the general torsion case is again the
non-trivial reduction to the zero torsion case, similarly to
\cite{IVO}, while the remaining parts of the argument are
identical to those in \cite{IVO}. Despite their priority these
results of \cite{IVO} are not mentioned in \cite{LW1}. Finally, in
the three dimensional case, the first complete proof of
\cite[Theorem~10]{LW1} is given in the earlier paper \cite{IV3}
while the divergence-free torsion case appeared in \cite{IVO},
both results are not cited in \cite{LW1}.  In fact, the earlier
paper \cite{IV3} is not referenced in \cite{LW1}.
\end{rmrk}

\begin{conv}
\label{conven} \hfill\break\vspace{-15pt}

\begin{enumerate}
\item[a)] We shall use $X,Y,Z,U$ to denote horizontal vector fields, i.e. $%
X,Y,Z,U\in H=Ker\, \theta$.

\item[b)] $\{e_1,\dots,e_{2n}\}$ denotes a local orthonormal basis of the
horizontal space $H$.

\item[c)] The summation convention over repeated vectors from the basis $%
\{e_1,\dots,e_{2n}\}$ will be used. For example, for a (0,4)-tensor $P$, the
formula $k=P(e_b,e_a,e_a,e_b)$ means
\begin{equation*}
k=\sum_{a,b=1}^{2n}P(e_b,e_a,e_a,e_b);
\end{equation*}
\end{enumerate}
\end{conv}

\textbf{Acknowledgments}  The research is partially supported by
Contract ``Idei", DID 02-39/21.12.2009. S.I is partially supported
by Contract 181/2011 with the University of Sofia `St.Kl.Ohridski'

\section{Pseudohermitian manifolds and the Tanaka-Webster connection}

In this section we will briefly review the basic notions of the
pseudohermitian geometry of a CR manifold. Also, we recall some results (in
their real form) from \cite{T,W,W1,L1}, see also \cite{DT,IVZ,IV2}, which we
will use in this paper.

A CR manifold is a smooth manifold $M$ of real dimension 2n+1,
with a fixed n-dimensional complex sub-bundle $\mathcal{H}$ of the
complexified tangent bundle $\mathbb{C}TM$ satisfying $\mathcal{H}
\cap \overline{\mathcal{H}}=0$ and $[
\mathcal{H},\mathcal{H}]\subset \mathcal{H}$. If we let $H=Re\,
\mathcal{H}\oplus\overline{\mathcal{H}}$, the real sub-bundle $H$
is equipped with a formally integrable almost complex structure
$J$. We assume that $M$ is oriented and there exists a globally
defined compatible contact form $\theta$ such that the
\emph{horizontal space} is given by ${H}=Ker\,\theta.$ In other
words, the hermitian bilinear form $$ 2g(X,Y)=-d\theta(JX,Y) $$ is
non-degenerate. The CR structure is called strictly pseudoconvex
if $g$ is a positive definite tensor on $H$. The vector field
$\xi$ dual to $\theta$ with respect to $g $ satisfying
$\xi\lrcorner d\theta=0$ is called the Reeb vector field. The
almost complex structure $J$ is formally integrable in the sense
that
\begin{equation*}
([JX,Y]+[X,JY])\in {H}
\end{equation*}
and the Nijenhuis tensor
\begin{equation*}
N^J(X,Y)=[JX,JY]-[X,Y]-J[JX,Y]-J[X,JY]=0.
\end{equation*}
A CR manifold $(M,\theta,g)$ with a fixed compatible contact form $\theta$
is called \emph{a pseudohermitian manifold}%
\index{pseudohermitian manifold}. In this case the 2-form
\begin{equation*}
d\theta_{|_{{H}}}:=2\omega
\end{equation*}
is called the fundamental form. Note that the contact form is determined up
to a conformal factor, i.e. $\bar\theta=\nu\theta$ for a positive smooth
function $\nu$, defines another pseudohermitian structure called
pseudo-conformal to the original one.

\subsection{The Tanaka-Webster connection}

The Tanaka-Webster connection \cite{T,W,W1} is the unique linear connection $%
\nabla$ with torsion $T$ preserving a given pseudohermitian
structure, i.e., it has the properties
\begin{equation}  \label{torha}
\begin{aligned}& \nabla\xi=\nabla J=\nabla\theta=\nabla g=0,\\
& T(X,Y)=d\theta(X,Y)\xi=2\omega(X,Y)\xi, \quad T(\xi,X)\in {H},
\\ & g(T(\xi,X),Y)=g(T(\xi,Y),X)=-g(T(\xi,JX),JY). \end{aligned}
\end{equation}
For a smooth function $f$ on a pseudohermitian manifold $M$ we denote by $\nabla
f $ its horizontal gradient,
\begin{equation}\label{e:hor grad}
g(\nabla f,X)=df(X).
\end{equation}
The horizontal sub-Laplacian $\triangle f$ and the norm of the horizontal
gradient $\nabla f =df(e_a)e_a$ of a smooth function $f$ on $M$ are defined
respectively by
\begin{equation}  \label{lap}
\triangle f\ =-\ tr^g_H(\nabla df)\ =\nabla^*df= -\ \nabla df(e_a,e_a),
\qquad |\nabla f|^2\ =\ df(e_a)\,df(e_a).
\end{equation}
The function $f\not\equiv 0$ is an eigenfunction of the sub-Laplacian if
\begin{equation}  \label{eig}
\triangle f =\lambda f,
\end{equation}
where $\lambda$ is a necessarily non-negative constant.

It is well known that the endomorphism $T(\xi,.)$ is the obstruction a
pseudohermitian manifold to be Sasakian. The symmetric endomorphism $T_\xi:{H%
}\longrightarrow {H}$ is denoted by $A$,
$$A(X,Y)\overset{def}{=}T(\xi,X,Y),$$ and is called \emph{the
(Webster) torsion of the pseudohermitian manifold or
pseudohermitian torsion.} It is a completely trace-free tensor of
type (2,0)+(0,2),
\begin{equation}  \label{tortrace}
A(e_a,e_a)=A(e_a,Je_a)=0, \quad A(X,Y)=A(Y,X)=-A(JX,JY).
\end{equation}
Let $R=[\bi,\bi]-\bi_{[,]}$ be the curvature of the Tanaka-Webster
connection. The pseudohermitian Ricci tensor $Ric$, the
pseudohermitian scalar curvature $S$ and the pseudohermitian Ricci
2-form $\rho$ are defined by
\begin{equation*}
Ric(C,B)=R(e_a,C,B,e_a), \quad S=Ric(e_a,e_a),\quad
\rho(C,B)=\frac12R(C,B,e_a,Ie_a), \quad C,B\in\Gamma(TM).
\end{equation*}
{As well known $\rho$, sometimes called the Webster Ricci tensor, is the (1,1)-part of $Ric$. In dimension
three we have $Ric(.,.)=\rho(J.,.)$.} We refer the reader to \cite{IVO}
for a quick summary using real expression of the well known
properties of the curvature $R$ of the Tanaka-Webster connection
established in \cite{W,W1,L1}, see also \cite{DT,IVZ,IV2}.

\subsection{The Ricci identities for the Tanaka-Webster connection}

We shall use repeatedly the following Ricci identities of order two and
three for a smooth function $f$, see also \cite{IVO,IV2},
\begin{equation}  \label{e:ricci identities}
\begin{aligned} & \nabla^2f (X,Y)-\nabla^2f(Y,X)=-2\omega(X,Y)df(\xi)\\ &
\nabla^2f (X,\xi)-\nabla^2f(\xi,X)=A(X,\nabla f)\\ & \nabla^3 f
(X,Y,Z)-\nabla^3 f(Y,X,Z)=-R(X,Y,Z,\nabla f) - 2\omega(X,Y)\nabla^2f
(\xi,Z)\\ &\nabla ^{3}f(X,Y,Z)-\nabla ^{3}f(Z,Y,X)=-R(X,Y,Z,\nabla
f)-R(Y,Z,X,\nabla f)-2\omega (X,Y)\nabla ^{2}f(\xi ,\ Z)\\ &\hskip 1.8in
-2\omega (Y,Z)\nabla ^{2}f(\xi ,X) +2\omega (Z,X)\nabla ^{2}f(\xi ,Y)
+2\omega (Z,X)A(Y,\nabla f) \\ & \nabla^3 f(\xi,X,Y)-\nabla^3
f(X,\xi,Y)=(\nabla_{\gr}A)(Y,X)-(\nabla_YA)(\nabla f,X)-\nabla^2 f(AX,Y)\\
&\nabla^3 f(X,Y,\xi)-\nabla^ 3 f (\xi,X,Y)=\nabla^2f (AX,Y)+\nabla^2f (X,AY)
+(\nabla_X A)(Y,\nabla f)+(\nabla_Y A)(X,\nabla f)\\ &\hskip4.4in
-(\nabla_{\nabla f}) A( X,Y). \end{aligned}
\end{equation}

An important consequence of  the first Ricci identity is the following fundamental formula
\begin{equation}  \label{xi1}
g(\nabla^2f,\omega)=\nabla^2f(e_a,Je_a)=-2n\,df(\xi).
\end{equation}
On the other hand, by \eqref{lap} the trace with respect to the metric is the negative sub-Laplacian
\[
g(\nabla^2f,g)=\nabla^2f(e_a,e_a)=-\triangle f.
\]

We also recall the horizontal divergence theorem allowing "integration by parts" \cite{T}. Let $(M,
g,\theta) $ be a pseudohermitian manifold of dimension $2n+1$. For a fixed
local 1-form $\theta$ the form
\begin{equation*} 
\vol=\theta\wedge\omega^{n}
\end{equation*}
is a globally defined volume form since $\vol$ is independent on the
local one form $\theta$. We define the (horizontal) divergence of a horizontal vector
field/one-form $\sigma\in\Lambda^1\, (H)$ defined by
\begin{equation*}
\nabla^*\, \sigma\ =-tr|_{H}\nabla\sigma=\ -(\nabla_{e_a}\sigma)e_a.
\end{equation*}
The divergence formula \cite{T} gives the "integration by parts" identity for a one-form of compact support
\begin{equation*}
\int_M (\nabla^*\sigma)\vol\ =\ 0.
\end{equation*}

\section{The vertical Bochner formula}
We recall explicitly the vertical Bochner formula from \cite{IVO}
since it will provide an  important step in the proof of the main
result. The proof below is contained in Remark 3.5 of \cite{IVO}.
\begin{prop} [\cite{IVO},"vertical Bochner
formula"]\label{remn=1}  For any smooth function $f$ on a pseudohermitian
manifold of dimension $(2n+1)$ the following identity holds
\begin{equation}  \label{e:vertical Bochner}
-\triangle (\xi f)^2 = 2|\nabla(\xi f)|^2-2df(\xi)\cdot \xi
(\triangle f) + 4df(\xi)\cdot g(A,\nabla^2 f) -4df
(\xi)(\nabla^*A)(\nabla f).
\end{equation}
\end{prop}
\begin{proof}
To prove \eqref{e:vertical Bochner}  we use the last of the Ricci
identities \eqref{e:ricci identities} and the fact that the
torsion is trace free to obtain
\begin{eqnarray*}
-\frac{1}{2}\triangle (\xi f)^{2} &=&\nabla ^{3}f(e_{a},e_{a},\xi
)df(\xi
)+\nabla ^{2}f(e_{a},\xi )\nabla ^{2}f(e_{a},\xi ) \\
&=&\left[ \nabla ^{3}f(\xi ,e_{a},e_{a})+2g(\nabla
^{2}f,A)-2(\nabla ^{\ast
}A)(\nabla f)\right] df(\xi )+|\nabla (\xi f)|^{2} \\
&=&|\nabla (\xi f)|^{2}-df(\xi )\cdot \xi (\triangle f)+2df(\xi
)\cdot g(A,\nabla ^{2}f)-2df(\xi )(\nabla ^{\ast }A)(\nabla f),
\end{eqnarray*}%
which completes the proof of \eqref{e:vertical Bochner}.
\end{proof}

\section{The Hessian of an extremal function in the extremal three dimensional case}

In this section we recall some results from \cite{IVO} determining
the full Hessian of an "\textit{extremal first eigenfunction}"
that is an eigenfunction with the smallest possible eigenvalue under the Lichnerowicz type condition.

Let $M$ be a compact strictly
pseudoconvex CR 3-manifold   satisfying
\begin{equation*}
Ric(X,X)+4A(X,JX)\geq k_0 \,g(X,X)
\end{equation*}
such that the CR-Paneitz operator is \emph{non-negative on} $f$.
If $\frac12k_0$ is an eigenvalue of the sub-Laplacian, $\triangle
f=\frac12k_0f$ then the corresponding eigenfunctions satisfy the
next identities, cf. Section 3 of \cite{IVO},
\begin{equation}  \label{eq7}
\nabla^2f(X,Y)=-\frac{k_0}{4}fg(X,Y)-df(\xi)\omega(X,Y).
\end{equation}
{Furthermore by Remark~3.2 of \cite{IVO}, we have}
\begin{equation}  \label{eq14}
Ric(\nabla f,\nabla f)+4A(J\nabla f,\nabla f)=k_0|\nabla f|^2,
\hskip.7in \int_M P_f(\gr)\vol=0.
\end{equation}
{Since the horizontal space is two dimensional we can use
$\gr$, $J\gr$ as a basis at the points where $|\gr|\not=0$.  In
fact, we have $|\gr|\not=0$ almost everywhere. This follows from
\cite[Lemma~5.1]{IVO} showing that $f$ satisfies a certain
elliptic equation which implies that $f$ cannot vanish on any open
set since otherwise $f\equiv 0$ which is a contradiction.}

{The "mixed" parts of the Hessian are given in the second
and the third equations in the proof of Theorem 5.2 in \cite{IVO}
as follows}
\begin{equation}\label{n12}
\begin{aligned}
\nabla^{2} f(\xi ,J\nabla f)& =-|\nabla f|^{2}+A(J\nabla f,\nabla f)=-\frac{1%
}{4}Ric(\nabla f,\nabla f)=-\frac{S}{8}|\nabla f|^{2},\\
 \nabla ^{2}f(\xi ,\nabla f)& =-\frac{1}{3}A(\nabla
f,\nabla f),\qquad \nabla ^{2}f(\nabla f,\xi )=\frac{2}{3}A(\nabla
f,\nabla f).
\end{aligned}
\end{equation}
{The Ricci identities together with \eqref{n12} imply}
\begin{equation}\label{n13}
\nabla^{2} f(J\nabla f, \xi) =-|\nabla f|^{2}+ 2A(J\nabla f,\nabla
f).
\end{equation}
Using a homothety we can reduce to the case $\lambda_1=2$ and
$k_0=4$, which are the values for the standard Sasakian round
3-sphere. Henceforth, we
shall work under these assumptions. Thus, for  an extremal first eigenfunction $f$ (by definition $%
f\not\equiv 0$)   we have the equalities
\begin{equation} \label{eq1}
\begin{aligned} \lambda = 2, \qquad \triangle f=2 f, \qquad
\int_M(\triangle f)^2\vol=2\int_M|\nabla f|^2\vol.
\end{aligned}
\end{equation}
In addition, the horizontal Hessian of $f$ satisfies \eqref{eq7},
which with the assumed normalization takes the form given in
equation \eqref{e:hessian}.

\section{Vanishing of the pseudohermitian torsion.}\label{ss:3D}

In this section we show the vanishing of the pseudo-hermitian
torsion. We shall assume, unless explicitly stated otherwise, that
$M$ is a  compact strictly pseudoconvex pseudohermitian CR
manifold of dimension three for which \eqref{condm} holds and $f$
is a smooth function on $M$ satisfying \eqref{e:hessian}. In
particular, we have done the normalization, if necessary, so that
\eqref{condm} holds with $k_0=4$.

\begin{lemma}\label{neww1}
Let $f$ be an extremal eigenfunction of the sublaplacian on a
compact strongly pseudoconvex 3-dimensional pseudohermitian
manifold. Then we have
\begin{equation}\label{new1}
A(\gr,\gr)=0
\end{equation}
and the "mixed" derivatives are given by
\begin{equation}
\label{e:vhessian}
 \nabla ^{2}f(\xi ,Y)=df(JY)+A(Y,\nabla f),\qquad \nabla ^{2}f(Y,\xi
)=df(JY)+2A(Y,\nabla f).
\end{equation}
\end{lemma}
\begin{proof}

Using the "vertical" Bochner formula \eqref{e:vertical Bochner}
and taking into account that $g(A,\nabla^2 f)=0$, after an
integration by parts  we obtain
\begin{multline}\label{e:applvB}
0=\int_M |\nabla(\xi f)|^2-df(\xi)\cdot \xi
(\triangle f) + 2df(\xi)\cdot g(A,\nabla^2 f) -2df
(\xi)(\nabla^*A)(\nabla f) \vol\\
=\int_M |\nabla(\xi f)|^2 -2(\xi\, f)^2-2\nabla^2 f(e_a,\xi)A(e_a,\gr)\vol\\
=\int_M  -2(\xi\, f)^2+\frac {1}{|\gr|^2} \Big[ \big (\nabla^2f(\gr,
\xi)\big)^2 +\big(\nabla^2 f(J\gr,
\xi)\big)^2\Big ]\vol\\
\int_M\frac {-2}{|\gr|^2} \Big[\nabla^2f(\gr,
\xi)A(\gr,\gr) + \nabla^2 f(J\gr,
\xi)A(J\gr,\gr)\Big ]\vol.
\end{multline}
using that
\[
\frac {1}{|\gr|^2} |A(\gr,\gr)| \leq  ||A||\overset{def}{=}\sup_M {|A|} \quad \text{a.e.}
\]
since $|\gr|\not=0$ almost everywhere.  Using \eqref{n12},  and
\eqref{n13} \eqref{e:applvB} takes the form
\begin{equation}\label{e:applvB2}
\int_M\Big[\frac {-8}{9|\gr|^2} \Big (A(\gr,\gr)\Big )^2 + |\gr|^2 -2|\xi f|^2-2A(J\gr,\gr)\Big ]\vol=0.
\end{equation}
Now, we recall \cite[Lemma~8.6]{IVO} and
\cite[Lemma~8.7]{IVO}  implying an identity which in the case $n=1$ reduces to
\begin{equation}\label{e:Afrom2lemmas}
2\int_M  A(J\nabla f,\nabla f)\vol =\int_M\Big[ -\frac {1}{2}g(\nabla^2 f,\omega)^2+\frac {1}{2}(\triangle f)^2+\frac {1}{2}P(\gr)\Big]\vol.
\end{equation}
Taking into account \eqref{xi1}, \eqref{eq1} and the fact that in
the  extremal case we have $\int_MP_f(\gr)=0$ by \eqref{eq14},
 \eqref{e:applvB2} and \eqref{e:Afrom2lemmas} imply
\begin{equation}
\frac{16}{9}\int_M \Big ( \frac {A(\gr,\gr)}{|\gr|}\Big)^2\vol=0,
\end{equation}
hence the claimed result for the torsion $A$. The formulas  for
the mixed derivatives follow taking also into account \eqref{n12},
 and \eqref{n13}.
\end{proof}
Lemma \ref{neww1} implies a number of crucial identities which we
record in the following Corollaries.
\begin{cor}\label{neww2}
The following identities holds true almost everywhere
\begin{align}\label{e:AJdf}
|\gr|^2A(JY,\gr)=df(Y)A(\gr, J\gr),\\\label{e:formulanormA1}
|\gr|^4 |A|^2=2\Big( A(\gr, J\gr)\Big )^2.
\end{align}
In addition, we have
\begin{equation}\label{e:formulanormA2}
|\gr|^2 |A|=-\sqrt{2} A(\gr, J\gr).
\end{equation}
\end{cor}
\begin{proof}
Since $\gr,J\gr$ form a basis of $H$ almost everywhere, then
\eqref{e:AJdf} follows from \eqref{new1} by a direct verification.
Then, \eqref{e:formulanormA1} follows since the horizontal space
is two dimensional.
 Notice
that Lichnerowicz' condition implies that
\begin{equation}
A(\gr, J\gr)\leq 0,
\end{equation}
which, together with \eqref{e:formulanormA1} imply
\eqref{e:formulanormA2}.
\end{proof}
The proof of \cite[Lemma 4.2]{IVO} shows that \eqref{e:vhessian} gives the following fact.
\begin{lemma}\label{l:D3f} Let $M$ be a strictly pseudoconvex pseudohermitian
CR manifold of dimension three. If $f$ is an eigenfunction
of the sub-Laplacian satisfying \eqref{e:hessian}, then the
following formula for the third covariant derivative holds true
\begin{equation}  \label{e:D3f extremal}
\nabla ^{3}f(X,Y,\xi ) =-df(\xi )g(X,Y)+f\omega
(X,Y)-2fA(X,Y)-2df(\xi )A(JX,Y) +2(\nabla_X A)(Y,\nabla f).
\end{equation}
\end{lemma}

We turn to the proof of our main result.
\begin{thrm}\label{t:A vansihes if div 3D}
Let $M$ be a compact strictly pseudoconvex
pseudohermitian CR manifold of dimension three for which the
Lichnerowicz condition \eqref{condm} holds and the Paneitz operator is non-negative. If  $%
f $ is an eigenfunction satisfying \eqref{e:hessian} then the
pseudohermitian torsion vanishes, $A=0$.
\end{thrm}

\begin{proof}
First we show
$$g(\gr,\nabla|A|^2) =0.$$
Indeed, Lemma \ref{l:D3f} gives
\begin{equation}\label{e:dfdA1}
g(A,\nabla^3f(.,.,\xi))=-2f|A|^2+2g(A,\nabla A(.,.,\gr)).
\end{equation}
Next we compute the above scalar product using the Ricci
identities. In fact, in the last Ricci identity  we make the
substitution
\[
\nabla^3 f(\xi, X,Y))=-df(\xi)g(X,Y)-(\xi^2 f)\omega(X,Y),
\]
which follows from the Hessian equation \eqref{eq7}, to obtain the
equation
\begin{multline}\label{e:ricci2}
\nabla^3 f(X,Y,\xi)=-df(\xi)g(X,Y)-(\xi^2 f)\omega(X,Y)-2f\cdot A(X,Y)\\
+ \nabla A(X,Y,\gr)+2df(\xi)A(JX,Y)-\nabla A(\gr,X,Y).
\end{multline}
Now, \eqref{e:ricci2} implies
\begin{equation}\label{e:dfdA2}
g(A,\nabla^3f(.,.,\xi))=-2f|A|^2+2g(A,\nabla
A(.,.,\gr))-g(A,\nabla A(\gr,.,.)).
\end{equation}
Equations \eqref{e:dfdA1} and \eqref{e:dfdA2} show that
\begin{equation}\label{e:dfdA3}
g(A,\nabla A(\gr,.,.))=0,\qquad
g(A,\nabla^3f(.,.,\xi))=-2f|A|^2+2g(A,\nabla A(.,.,\gr))
\end{equation}
hence
\begin{equation}\label{e:dfdA4}
g(\gr,\nabla|A|^2) = 2g(A,\nabla A(.,.,\gr))=0.
\end{equation}
Equation \eqref{e:dfdA4} implies that for any $k>0$ we have
\begin{equation}\label{e:dfdAkey}
g(\gr,\nabla|A|^k) =0.
\end{equation}
For simplicity, suppose $A\not=0$ everywhere.  For the complete
argument, which requires the introduction of a cut-off function we
refer to \cite[Lemma 4]{LW}).  Since $M$ is compact there is a
constant $a>0$ such that at every point of $M$ we have
\begin{equation}\label{e:Abelow}
a<|A|, \quad \text{hence}\quad |A|^2\leq\frac {1}{a}|A|^3.
\end{equation}
The divergence formula gives
\begin{multline}\label{e:A3calc0}
\int_M |A|^3f^{2(k+1)}\vol=-\frac {1}{2}\int_M |A|^3f^{2k+1}\triangle f \vol=
\frac {1}{2}\int_M g(\nabla (|A|^3f^{2k+1}),\gr) \vol\\
=\frac {2k+1}{2}\int_M |A|^3f^{2k}|\gr|^2 \vol + \frac {1}{2}\int_M f^{2k+1} g(\nabla |A|^3,\gr) \vol\\
=\frac {2k+1}{2}\int_M |A|^3f^{2k}|\gr|^2 \vol,
\end{multline}
taking into account \eqref{e:dfdAkey}.

With the help of  \eqref{e:formulanormA2}, the divergence
formula  and \eqref{e:dfdAkey} we can compute the last integral as
follows
\begin{multline}\label{e:A3calc1}
\sqrt{2}\int_M |A|^3f^{2k}|\gr|^2 \vol = -\int_M |A|^2f^{2k}A(\gr,J\gr) \vol=
-\int_M |A|^2f^{2k}A(e_a,J\gr)df(e_a) \vol\\
=\int_M g(\nabla |A|^2,AJ\gr)f^{2k+1}\vol +2k\int_M |A|^2
f^{2k}A(\gr,J\gr)\vol\\ + \int_M |A|^2f^{2k+1}\nabla
A(e_a,e_a,J\gr)\vol +\int_M
|A|^2f^{2k+1}A(e_a,J\nabla_{e_a}(\gr))\vol.
\end{multline}
The last integral equals zero due to \eqref{e:hessian}. The first
integral is zero due to \eqref{e:AJdf} and \eqref{e:dfdAkey}.
Therefore, the first and last equality in \eqref{e:A3calc1} give
\begin{equation}\label{e:A3calc2}
\sqrt{2}(2k+1)\int_M |A|^3f^{2k}|\gr|^2 \vol =\int_M
|A|^2f^{2k+1}\nabla A(e_a,e_a,J\gr)\vol .
\end{equation}
Therefore, using \eqref{e:Abelow} and H\"older's inequality we
have
\begin{multline}
\sqrt{2}(2k+1)\int_M |A|^3f^{2k}|\gr|^2 \vol \leq ||\nabla^* A||\int_M |A|^2f^{2k+1}|\gr|\vol \leq \frac {||\nabla^* A||}{a}\int_M |A|^3f^{2k+1}|\gr|\vol\\
\leq \frac {||\nabla^* A||}{a}\Big (\int_M |A|^3f^{2(k+1)}\vol\Big)^{1/2}\, \Big( \int_M |A|^3f^{2k}|\gr|^2\vol \Big)^{1/2}\\
=\frac {||\nabla^* A||}{a}\Big (\int_M
|A|^3f^{2(k+1)}\vol\Big)^{1/2}\, \Big( \frac{2}{2k+1}\int_M
|A|^3f^{2(k+1)}\vol \Big)^{1/2},
\end{multline}
using \eqref{e:A3calc0} in the last equality. Now, equation
\eqref{e:A3calc0} gives
\[
\sqrt{2}(2k+1)\int_M |A|^3f^{2k}|\gr|^2 \vol \leq \frac
{||\nabla^* A||}{a}\Big( \frac{2k+1}{2} \Big)^{1/2}\Big (\int_M
|A|^3f^{2k}|\gr|^2 \vol \Big),
\]
which gives a contradiction by taking $k$ sufficiently large.
Therefore the torsion vanishes, $|A|=0$.

\end{proof}

\subsection{Proof of  Theorem \ref{main2}}  We  apply Theorem \ref{t:A vansihes if div 3D} to conclude
that the pseudohermitian torsion vanishes. The claim of the
Theorem~\ref{main2} follows by applying the known result in the
torsion-free (Sasakianan) case.

\end{document}